\documentclass[graybox]{SNmult}

\usepackage{amsfonts}
\usepackage{graphicx}
\usepackage{epstopdf}
\usepackage{algorithmic}
\usepackage{amsopn}
\usepackage{amsmath} 
\usepackage{bm} 

\usepackage{bm}
\usepackage{xspace}

\newcommand{\mr}[1]{\ensuremath{\mathrm{#1}}}

\newcommand{\fnc}[1]{\ensuremath{\mathit{#1}}}
\newcommand{\bfnc}[1]{\ensuremath{\bm{\mathit{#1}}}}

\newcommand{\U}[0]{\ensuremath{\bfnc{U}}}

\newcommand{\Efnc}[0]{\ensuremath{\fnc{E}}}

\usepackage{tikz}
\usepackage{mathrsfs}
\usepackage{subcaption}
\usepackage{xcolor}

\title*{Implicit BDF2 dual time-stepping positivity-preserving entropy-stable schemes for unsteady compressible viscous flows}
\titlerunning{Implicit BDF2 DTS positivity-preserving scheme for unsteady viscous flows}
\author{Mohammed Sayyari \and Nail K. Yamaleev}
\institute{Mohammed Sayyari and  Nail K. Yamaleev \at Department of Mathematics and Statistics, Old Dominion University, Norfolk, VA, \email{malsayya@odu.edu and nyamalee@odu.edu}}

\begin{document}
\maketitle

\abstract*{
  This paper presents a rigorous extension of the explicit, high-order, positivity-preserving, and entropy-stable spectral collocation schemes developed in ~\cite{upperman2023first,yamaleev2023high} for the 3D compressible Navier-Stokes equations to a time-implicit formulation. The time derivative terms are discretized by using the second-order implicit backward difference formula (BDF2) that is well suited for solving time-variable viscous flows at high Reynolds numbers. The nonlinear system of discrete equations resulting from the BDF2 discretization at each physical timestep is solved using a dual time-stepping (DTS) technique. The BDF2 DTS scheme is entropy-stable and positivity-preserving in the pseudotime and provides unconditional stability properties in the physical time. Numerical results demonstrate the efficiency and accuracy of the positivity-preserving BDF2 DTS scheme as compared with its explicit counterpart are presented for supersonic flows with strong shock waves and contact discontinuities. }

\section{Introduction}
\label{sec:introduction}

Hypersonic turbulent regimes are characterized by high temperatures and strong compressibility effects. These effects include shocks and contact discontinuities as primary sources of instability that can be remidied by entropy-stbility of compressible Navier--Stokes discretizations for admissible thermodynamic states. In high-order approximations, unresolved flow features can lead to undershoots and, in turn, violate the positivity of density and temperature required for entropy-stability. Thus, enforcing positivity is a prerequisite for entropy-stability. This motivates the development of positivity-preserving and entropy-stable schemes for practical shock-dominated simulations. Explicit Runge--Kutta (RK) positivity-preserving entropy-stable spectral collocation schemes for the compressible Navier-Stokes equations developed in~\cite{upperman2023first,yamaleev2023high} address these challenges. However, these explicit schemes suffer from the Courant-Friedrichs-Lewy (CFL)-type time step constraint, which becomes very restrictive due to the grid stiffness, thus significantly increasing the computational cost. This paper presents an implicit dual time-stepping (DTS) second-order backward difference formula (BDF2) positivity-preserving entropy-stable high-order spectral collocation scheme for solving the 3D compressible Navier--Stokes equations. These schemes provide unconditional stability in the physical time, and enjoy higher bounds in the explicit pseudotime, which allows for the increase of the overall efficiency as compared to the explicit counterpart developed in \cite{upperman2023first,yamaleev2023high}. Implicit DTS positivity-preserving entropy-stable high-order spectral collocation schemes were introduced in~\cite{sayyari2026implicit} for solving the 3D compressible Navier--Stokes equations centering the derivation on the BDF1 DTS scheme. This paper extends this development by presenting the full derivation and proof of positivity of thermodynamic variables for the positivity-preserving BDF2 DTS scheme. Then, the efficiency and accuracy of the BDF2 DTS scheme are assessed for unsteady viscous flows at high Mach and Reynolds numbers.


The paper is organized as follows. The regularized Navier--Stokes equations is defined in Section \ref{sec:rns}. The BDF2 DTS scheme is presented in \ref{sec:dual_time-stepping}. Then, the positivity-preserving semi-discrete operators are briefly described in Section \ref{sec:semi-discrete}. Lastly, the efficiency and accuracy of the BDF2 DTS and the explicit, positivity-preserving, and strong stability preserving (SSPRK3) scheme are compared in Section \ref{sec:results}.


\section{The regularized Navier--Stokes equations}
\label{sec:rns}


The compressible Navier--Stokes equations has no theoretical results showing that they intrinsically preserve the positivity of thermodynamic variables. In contrast, the Brenner--Navier--Stokes equations have been shown to guarantee global-in-time positivity of the thermodynamic variables~\cite{feireisl2010new}. To achieve this property, the approach introduced in~\cite{upperman2023first,yamaleev2023high} is followed. The Navier--Stokes equations are regularized by including artificial dissipation in the form of the diffusion operator of the Brenner-Navier-Stokes equations. The regularized Navier-Stokes (RNS) equations in curvilinear coordinates $(\xi_1,\xi_2, \xi_3)$ are given by
\begin{align}
    \label{eq:brenner-ad-Navier--Stokes}
    \frac{\partial (JU)}{\partial t}
    + \frac{\partial}{\partial \xi_d}
    \left[
        J \frac{\partial \xi_d}{\partial x_i}
    \right]
    {\bm F}^{(I)}_i
    = \frac{\partial}{\partial \xi_d}
    \left[
        J \frac{\partial \xi_d}{\partial x_i}
    \right]
    \left(
        {\bm F}^{(V)}_i
        + {\bm F}^{(AD)}_i
    \right),
\end{align}
where $\U = \left[\rho,\rho v_1,\rho v_2,\rho v_3,\rho\Efnc\right]^\top$ is a vector of the conservative variables ($\rho$- density, $\rho v_i$- momentum, and $\rho E$- total energy), $J=\left|\frac{\partial(x_1, x_2, x_3)}{\partial(\xi_1, \xi_2, \xi_3)} \right|$ is the metric Jacobian, and $(x_1, x_2, x_3)$ are the Cartesian coordinates. In Equation \eqref{eq:brenner-ad-Navier--Stokes} and hereinafter, the Einstein summation convention is used $\cdot_d\cdot_d = \sum_{d}\cdot_d\cdot_d$. The inviscid and viscous fluxes of the Navier--Stokes equations, ${\bm F}^{(I)}_i$ and ${\bm F}^{(V)}_i$, are given by
\begin{align}
    \label{eq:inviscid-viscous-fluxes}
    &{\bm F}^{(I)}_i
    = \begin{bmatrix}
        \rho v_i \\
        \rho v_1 v_i + P\delta_{1i} \\
        \rho v_2 v_i + P\delta_{2i} \\
        \rho v_3 v_i + P\delta_{3i} \\
        (\rho\Efnc + P)v_i
    \end{bmatrix},
    &&
    {\bm F}^{(V)}_i
    = \begin{bmatrix}
        0 \\
        \tau_{1i} \\
        \tau_{2i} \\
        \tau_{3i} \\
        \tau_{ji}v_j 
        + \kappa \frac{\partial T}{\partial x_i}
    \end{bmatrix},
\end{align}
where $\delta_{ij}$ is the Kronecker delta, $P=\rho R_g T$ is the pressure, $R_g$ is the gas constant, $T$ is the temperature, and $\kappa$ is the heat conductivity. The stress tensor is given by
\begin{equation}
    \nonumber
    \tau_{ji} = \mu \left( \frac{\partial v_j}{\partial x_i} + \frac{\partial v_i}{\partial x_j} - \frac{2}{3}\delta_{ji}\frac{\partial v_l}{\partial x_l}\right),
\end{equation}
where $\mu$ is the dynamic viscosity. The artificial dissipation flux, ${\bm F}^{(AD)}_i$, is defined as follows:
\begin{align}
    {\bm F}^{(AD)}_i
    &= {\bm F}^{(V)}_i\Big|_{\substack{\mu=c_{\rho}\mu^{AD}\\ \kappa = c_T\mu^{AD}}}
        + \sigma
        \begin{bmatrix}
            1 \\ v_1 \\ v_2 \\ v_3 \\ \Efnc 
        \end{bmatrix}
        \frac{\partial \rho}{\partial x_i}
\end{align}
where $\sigma = c_\rho\mu^{AD}/\rho$ and $\mu^{AD}$ is an artificial dissipation coefficient. For all test problems presented herein, the tunable coefficients, $c_\rho$ and $c_T$, are set equal to $0.9$ and $\frac{c_\rho}{\gamma-1}$, respectively. 

Similar to the compressible Navier--Stokes equations, the RNS equations is equipped with the Godunov entropy-pair $(\mathcal{S}, \mathcal{F}) = (-\rho s, -\rho s \bm{v})$. Assuming the corresponding boundary conditions are entropy-stable, the RNS equations satisfy the following entropy inequality~\cite{yamaleev2019entropy}:
\begin{equation}\label{eq:EinEquation3}
  \int_{\hat\Omega} \frac{\partial ( J\mathcal{S})}{\partial t}\mr{d}\hat\Omega
 = \frac{d}{d t} \int_{\hat\Omega} J \mathcal{S}\mr{d}\hat\Omega
  \leq 0.
\end{equation} 
Note that along with the entropy inequality given by Equation~\eqref{eq:EinEquation3}, the regularized Navier--Stokes equations~\eqref{eq:brenner-ad-Navier--Stokes} preserve the positivity of thermodynamic variables.


\section{The dual time-stepping method}
\label{sec:dual_time-stepping}

The time derivative terms of most positivity-preserving entropy stable schemes, such as the one developed in \cite{upperman2023first,yamaleev2023high}, are discretized by using explicit SSP Runge-Kutta methods. However, these explicit time integrators impose a CFL-type condition on the time step, which becomes stiff for high-Reynolds-number flow simulations. To eliminate this stiffness and preserve the positivity properties, a DTS technique based on the implicit BDF2 scheme is used. The implicit BDF2 scheme is A-stable for nonlinear problems and thus is well-suited for simulating stiff unsteady viscous flows at high Reynolds numbers. The BDF2 DTS scheme for \eqref{eq:brenner-ad-Navier--Stokes} can be written in the following semi-discrete form:
\begin{equation}\label{eq:pseudostep}
  \frac{\partial \hat{\bf u}^*}{\partial\tau} 
  = - \frac{3\hat{\bf u}^{*} - 4\hat{\bf u}^n + \hat{\bf u}^{n-1}}{2 \Delta t} 
  + {\bf R}^*,
\end{equation}
where $\hat{\bf u}^*$ is a semi-discrete steady-state solution in the pseudotime, the ${\bf R}^*$ term represents the spatial discretization used evaluated at $\hat{\bf u}^*$, $\Delta t$ is a physical time step size, and $n$ is the current physical time level. When the dual time derivative converges to zero, i.e., $(\hat{\bf u}^*)_{\tau} \rightarrow {\bf 0}$, $\hat{\bf u}^*$ converges to ${\hat{\bf u}}^{n+1}$ and Equation~\eqref{eq:pseudostep} becomes the standard BDF2 scheme. The update in pseudotime is achieved by taking a forward Euler step in the $\Delta \tau_k$ direction
\begin{equation}\label{eq:semi-implicit-update}
  \hat{\bf u}^{k+1} 
  = \hat{\bf u}^k 
  + \Delta\tau_k \left(
    - \frac{3\hat{\bf u}^{k+1} - 4\hat{\bf u}^n + \hat{\bf u}^{n-1}}{2 \Delta t} 
    + {\bf R}^k
  \right),
\end{equation}
where $\hat{\bf u}^{k}=[J]{\bf u}^{k}$, $\hat{\bf u}^{n}=[J]{\bf u}^{n}$ and $[J]$ is a diagonal matrix composed out of the metric Jacobian computed at the corresponding solution points. The explicit update formula is then obtained by writing the formula for $\hat{\bf u}^{k+1}$
\begin{equation}
 \label{eq:bdf2-DTS-update}
  \hat{\bf u}^{k+1}
  = \frac{2\Delta t_n}{2\Delta t_n+3\Delta \tau_k}
  \left(
    \hat{\bf u}^k 
    + 2 \frac{\Delta \tau_k}{\Delta t_n} \hat{\bf u}^n 
    - \frac12 \frac{\Delta \tau_k}{\Delta t_n} \hat{\bf u}^{n-1}
    + \Delta\tau_k{\bf R}^k
  \right).
\end{equation}
This update formula is used to converge the solution to the steady-state in the pseudotime, 
which can be interpreted as an iterative solver for solving the nonlinear discrete equations at each physical time step.

For the sake of brevity, the subscripts $k$ in $\Delta \tau_k$ is omitted and the following notation is used:
\begin{equation}
\label{eq:Ctau}
  \hat{\bf u}^{k,n} = \hat{\bf u}^k + 2 \frac{\Delta \tau}{\Delta t} \hat{\bf u}^n - \frac12 \frac{\Delta \tau}{\Delta t} \hat{\bf u}^{n-1};
  \quad C_2^{\tau} = \frac{2\Delta t}{2\Delta t+3\Delta \tau}.
\end{equation}
With this notation, the BDF2 DTS scheme given by Equation~\eqref{eq:bdf2-DTS-update} becomes
\begin{equation}
 \label{eq:bdf2-update}
 \hat{\bf u}^{k+1}
 = C_2^{\tau}\left(\hat{\bf u}^{k,n} + \Delta\tau{\bf R}^k
 \right).
\end{equation}


\section{Spatial discretization}
\label{sec:semi-discrete}

This section outlines the baseline semi-discrete first-order positivity-preserving entropy-stable finite volume (FV) scheme for the regularized Navier--Stokes equations (\ref{eq:brenner-ad-Navier--Stokes}) discretized on high-order hexahedral Legendre-Gauss-Lobatto (LGL) grids, introduced in \cite{upperman2023first} and a similar notation is utilized to ease the reference to this baseline scheme. The spatial and temporal operators in this baseline scheme are extended to attain high-order positivity-preserving properties and unconditional stability in the physical time for the BDF2 DTS scheme.

\subsection{Spatial operators}
\label{sec:sbp-operators}

The spatial derivatives in the regularized Navier-Stkes equations \eqref{eq:brenner-ad-Navier--Stokes} are discretized using summation-by-parts (SBP) operators to utilize their mimetic properties. The one-dimensional (1D) SBP operators used herein are briefly outlined and further details on how these operators can be generalized to fully unstructured grids in multiple spatial dimensions are provided in \cite{fernandez2014review}.  The physical domain is divided into $N_{\rm elem}$ non-overlapping discontinuous elements, $[x_1^j, x_{N_p}^j]$, such that $x_1^j=x_{N_p}^{(j-1)}$ for $j=2,\dots,N_{\rm elem}$. The solution of order $p$ in each cell is approximated on $N_p=p+1$ LGL points, ${\bf x}^j = \left[x_1^j, \dots, x_{N_p}^j \right]^\top$ (referred to as solution points) for $j=1,\dots,N_{\rm elem}$. This representation provides us with a set of operators, including a quadrature, $\mathcal{P}$, a 1st-derivative differentiation operator, $\mathcal{D}$, and a stiffness matrix, $\mathcal{Q}$. Only the diagonal-norm LGL operators are considered herein. The main properties of these operators for a fixed order $p$ are as follows.
\begin{enumerate}
  \item For any vector ${\bf x}^l = \left[x_1^l,\dots, x_{N_p}^l\right]^\top$ and powers $l=0, 1,\dots, p$, $\mathcal{D}{\bf x}^l = \mathcal{P}^{-1}\mathcal{Q}{\bf x}^l = l{\bf x}^{l-1}$.
  \item $\mathcal{P}$ is a symmetric positive definite (SPD) matrix.
  \item $\mathcal{Q}+\mathcal{Q}^\top=\mathcal{B}$, where $\mathcal{B} = {\rm diag}(-1,0,\dots,0,1)$.
\end{enumerate}

Along with the solution points, an additional set of intermediate points is used, $\bar{\bf x}^j = \left[\bar{x}_0^j, \dots, \bar{x}_{N_p}^j \right]^\top$ for $j=1,\dots,N_{\rm elem}$. These points, which are referred to as flux points, form a complementary grid whose spacing is equal to the diagonal elements of the positive definite mass matrix $\mathcal{P}$, i.e., 
\begin{equation}\label{eq:fluxpoints}
\bar{x}_{i} - \bar{x}_{i-1} = \mathcal{P}_{ii} \ {\rm for} \ i=1,\dots, N_p. 
\end{equation}
The flux points are instrumental for constructing the first-order positivity-preserving entropy-stable scheme defined on high-order LGL elements, which will be discussed in Section \ref{sec:first-order}. As has been proven in \cite{fisher2013discretely}, any 1D SBP discrete differentiation operator $\mathcal{D} = \mathcal{P}^{-1} \mathcal{Q}$ presented above can be recast into the following telescopic flux form:
\begin{equation}
\nonumber
  \mathcal{P}^{-1} \mathcal{Q} {\bf f} 
  = \mathcal{P}^{-1} \Delta \bar{\bf f},
\end{equation}
where $\Delta$ is a $N_p\times (N_p +1)$ matrix corresponding to the two-point backward difference operator, and $\bar{\bf f}$ is a $p$th-order flux vector defined at the flux points~\cite{fisher2013discretely,carpenter2014entropy,fisher2011boundary}. 

Hereafter, the multidimensional SBP operators defined in the computational domain are denoted with subscripts $\cdot_{d}$, where $d$ is the $d$-th computational coordinate for $d = 1, 2, 3$. Because the scheme is developed for three spatial dimensions, with each coordinate defined by the index $d$, and there are $N_p$ LGL points in each direction, the notation $i_d$ is used to denote the $i$-th LGL point and $\overline{i}_d$ to denote the $i$-th flux point in the $d$-th curvilinear coordinate. The notation $\cdot_{i_1,i_2,i_3}$ is used to refer to a specific quantity at a point, where $i_1$, $i_2$, and $i_3$ are the indices of the quantity in the first, second, and third coordinates, respectively.


\subsection{First-order positivity-preserving entropy-stable scheme}
\label{sec:first-order}

The first-order entropy stable scheme is constructed in a finite volume manner on the high-order LGL solution points to guarantee the positivity of thermodynamic variables in the presence of strong discontinuities. The flux points act as control volume edges and can be written in the semi-discrete form as follows:
\begin{equation}\label{eq:first-order} 
  (\hat{{\bf u}}_1)_t
  = 
  \left(
  -[\mathcal{P}^{-1} \Delta]_{d} \left[ 
  \hat{\bar{{\bf f}}}^{1(I)}_d 
  - \hat{\bar{{\bf f}}}^{1(AD)}_{\hat{\bar{\sigma}},d}
  - \hat{\bar{{\bf f}}}^{1(AD)}_d 
  \right] 
  + \mathcal{D}_{d}\hat{{\bf f}}^{p(V)}_d \right)
  + 
  \mathcal{P}^{-1}_{d} 
  \hat{{\bf g}}^1_d, 
\end{equation} 
where $\hat{\bf u}_1= [J] {\bf u}_1$, $[J]$ is diagonal matrix composed out of the metric Jacobian evaluated at the corresponding solution points, $\hat{\bar{{\bf f}}}^{1(I)}_d$, $\hat{\bar{{\bf f}}}^{1(AD)}_{\hat{\bar{\sigma}},d}$, and $\hat{\bar{{\bf f}}}^{1(AD)}_d$ are first-order inviscid and artificial dissipation fluxes, $\hat{{\bf f}}^{p(V)}_d$ is a high-order physical viscous flux associated with the $d$-th coordinate, and $\hat{{\bf g}}^1_d$ represents inviscid, viscous, and artificial dissipation penalties \cite{upperman2023first}. Note that the discretization of the first-order inviscid fluxes on high-order LGL elements satisfies the geometric conservation law (GCL) equations \cite{upperman2023first,thomas1979geometric}. Further details on the construction of the fluxes and penalties in Equation~\eqref{eq:first-order} can be found in \cite{upperman2023first}. 

The fully discrete form of the first-order positivity-preserving entropy-stable BDF1 DTS scheme is obtained by substituting the right-hand side of Equation~\eqref{eq:first-order} into Equation~\eqref{eq:bdf2-update}, leading to the following update formula:
\begin{equation}
 \label{eq:bdf2-update-first}
 \hat{\bf u}_1^{k+1}
 = C_2^{\tau}\left(\hat{\bf u}_1^{k,n} + \Delta\tau{\bf R}_1^k
 \right).
\end{equation}


\subsection{The positivity of thermodynamic variables}
\label{sec:first-order-positivity}

The positivity of density and internal energy can only be shown for the first-order scheme~\eqref{eq:first-order}. The coefficient $C_2^{\tau}$ defined in Equation~\eqref{eq:Ctau} is always positive, thus, the positivity of density can be proven for the following quantity:
\begin{equation}
  \label{eq:positivity-density}
  \frac{\hat{\rho}_1^{k+1}}{C_2^{\tau}} 
  = \hat{\rho}_1^{k,n} + \Delta\tau{\bf R}_1^{\rho,k},
  \quad\quad
  {\bf R}_1^{\rho,k} =
  - \frac{
    \hat{\bar{f}}^{\rho+}_{i_d}
    - \hat{\bar{f}}^{\rho-}_{i_d}
  }{\Delta \bar{\xi}_{i_d}},
\end{equation}
provided that $\hat{\rho}_1^{k,n}$ is in the admissible set, where $\Delta \bar{\xi}_{i_d}=\bar{\xi}_{i_d+1}-\bar{\xi}_{i_d}$ is the distance between the neighboring flux points in the computational domain. The numerical fluxes $\hat{\bar{f}}^{\rho\pm}_{i_d}$ are defined as follows:
\begin{equation}\label{eq:flux}
 \hat{\bar{f}}^{\rho\pm}_{i_d}
 = \hat{\bar{m}}_{i_d}^{\pm} - \mathscr{D}_{i_d}^{\pm} \Delta_{i_d}^{\pm} {\rho},
\end{equation}
where $\Delta_{i_d}^+{\rho}={\rho}_{i_d+1}-{\rho}_{i_d}$ and $\Delta_{i_d}^-{\rho}={\rho}_{i_d}-{\rho}_{i_d-1}$, $\hat{\bar{m}}_{i_d}^{+}$ and $\hat{\bar{m}}_{i_d}^{-}$ are the momentums associated with the entropy conservative flux based on ${\bf u}_1(\xi_{i_d})$, ${\bf u}_1(\xi_{i_d+1})$ and  ${\bf u}_1(\xi_{i_d-1})$, ${\bf u}_1(\xi_{i_d})$, respectively, $\mathscr{D}_{i_d}^{\pm}$ is the corresponding dissipation coefficient whose minimum value is given by $\mathscr{D}_{i_d, \min}^{\pm} = \frac{|\hat{\bar{m}}^{\pm}_{i_d}|}{2{\rho}_{i_d,A}^{\pm}}$, with ${\rho}_{i_d,A}^{\pm}=({\rho}_{i_d}+{\rho}_{i_d\pm 1})/2$ (further details are available in~\cite{upperman2023first}).
\begin{theorem}
  If the $\hat{\bar{f}}^{\rho\pm}_{i_d}$ flux is defined by Equation~(\ref{eq:flux}) with $\mathscr{D}^{\pm}_{i_d} \geq \mathscr{D}^{\pm}_{{i_d},\min}= \frac{|\hat{\bar{m}}^{\pm}_{i_d}|}{2{\rho}_{i_d,A}^{\pm}}$, then the first-order BDF2 DTS scheme given by Eq.~(\ref{eq:bdf2-update}) preserves the positivity of density under the following constraint on $\Delta \tau$: 
  \begin{equation}
    \label{eq:density-BDF2}
    \Delta\tau < 
    \min\limits_{i_1, i_2, i_3}\frac{1}
    {\frac 2{J_{i_1i_2i_3}} 
    \frac{\mathscr{D}^{+}_d + \mathscr{D}^{-}_d}{\Delta \bar{\xi}_d}
    - \left(
      \frac{2}{\Delta t} \frac{\rho_1^{n}}{\rho_1^k}
      - \frac{1}{2\Delta t} \frac{\rho_1^{n-1}}{\rho_1^k}
    \right)}
    = \Delta\tau^{\rho}.
  \end{equation}
\end{theorem}
\begin{proof}
  The update formula~\eqref{eq:positivity-density} is split into the sum of positive and negative flux contributions in three dimensions {resulting in six flux terms (two for each spatial dimension). Thus, $\hat{\rho}_1^{k,n}$ is split into six contributions and added to each flux term. Since the flux contributions are repeated per dimension, only one representative dimension is shown explicitly here} as follows:
  \begin{equation}
    \label{eq:positivity-density-split}
    \frac{\hat{\rho}_1^{k+1}}{C_2^{\tau}}
    = \left(
        {\frac{\hat{\rho}_1^{k,n}}{6}} 
        - \Delta\tau
        \frac{
            \hat{\bar{f}}^{\rho+}_{i_d}
        }{\Delta \bar{\xi}_{i_d}}
    \right)
    + \left(
        {\frac{\hat{\rho}_1^{k,n}}{6}}
        + \Delta\tau
        \frac{
            \hat{\bar{f}}^{\rho-}_{i_d}
        }{\Delta \bar{\xi}_{i_d}}
    \right)
    + \cdots.
  \end{equation}
  The positive and negative flux contributions can be written as:
  \begin{equation*}
    \left(
      \frac{\hat{\rho}^{k,n}}{6} 
      \mp \Delta\tau\frac{\hat{\bar{f}}^{\rho\pm}_{i_d}}{\Delta \bar{\xi}_{i_d}}
    \right) = 
    \frac{\hat{\rho}^{k,n}}{6} 
    \mp \Delta\tau\frac{1}{\Delta \bar{\xi}_{i_d}}
    (\hat{\bar{m}}^\pm_{i_d} - \mathscr{D}^\pm_{i_d}\Delta^\pm_{i_d}\rho).
  \end{equation*}
  The momentum terms are scalars and thus {can} be bounded by $\mp\hat{\bar{m}}^\pm_{i_d} \geq -|\hat{\bar{m}}^\pm_{i_d}|$. Now, by using the definition of $\mathscr{D}^{\pm}_{i_d} \geq \frac{|\hat{\bar{m}}^{\pm}_{i_d}|}{2{\rho}_{1,i_d,A}^{\pm}}$, the positive and negative flux contributions can be bounded as follows:
  \begin{align}
  \nonumber
    \begin{split}
      &\frac{\hat{\rho}^{k,n}}{6} 
      \mp \Delta\tau\frac{1}{\Delta \bar{\xi}_{i_d}}
      (\hat{\bar{m}}^\pm_{i_d} - \mathscr{D}^\pm_{i_d}\Delta^\pm_{i_d}\rho)
      \geq \frac{\hat{\rho}^{k,n}}{6} 
      - \Delta\tau\frac{\mathscr{D}^\pm_{i_d}}{\Delta \bar{\xi}_{i_d}}
      \left(\frac{|\hat{\bar{m}}^\pm_{i_d}|}{\mathscr{D}^\pm_{i_d}} \mp \Delta^\pm_{i_d}\rho\right)\\
      &\hspace{14em}
      \geq \frac{\hat{\rho}^{k,n}}{6} 
      - \Delta\tau\frac{\mathscr{D}^\pm_{i_d}}{\Delta \bar{\xi}_{i_d}}
      \left(2\rho^\pm_{j,A} \mp \Delta^\pm_{i_d}\rho\right)\\
      &\hspace{14em}
      = \frac{\hat{\rho}^{k,n}}{6} 
      - \Delta\tau\frac{2\mathscr{D}^\pm_{i_d}}{\Delta \bar{\xi}_{i_d}}
      \rho^k\\
      &= \rho^k \left[
        \frac{J_{i_1i_2i_3}}{6} - \Delta\tau\left(
          \frac{2\mathscr{D}^+_{i_d}}{\Delta \bar{\xi}_{i_d}}
          - \frac{J_{i_1i_2i_3}}{6} 
            \left[
              \frac{2}{\Delta t} \frac{\rho_1^{n}}{\rho_1^k}
              - \frac{1}{2\Delta t} \frac{\rho_1^{n-1}}{\rho_1^k}
            \right]
        \right)  
      \right],
    \end{split}
  \end{align}
  where $2\rho^\pm_{j,A} \mp \Delta^\pm_{i_d}\rho = \rho_{i_d} + \rho_{i_d\pm1} - (\rho_{i_d\pm1} - \rho_{i_d}) = 2\rho_{i_d}^k$. Summing the six contributions in~\eqref{eq:positivity-density-split} gives the following bound on $\Delta\tau$:
  \begin{equation*}
    \frac{\hat{\rho}_1^{k+1}}{C_2^{\tau}}
    \geq \rho^k \left[
      J_{i_1i_2i_3} - \Delta\tau\left(
        2\frac{\mathscr{D}^+_{i_d} + \mathscr{D}^-_{i_d}}{\Delta \bar{\xi}_{i_d}}
        - J_{i_1i_2i_3}
          \left[
            \frac{2}{\Delta t} \frac{\rho_1^{n}}{\rho_1^k}
            - \frac{1}{2\Delta t} \frac{\rho_1^{n-1}}{\rho_1^k}
          \right]
      \right)  
    \right] > 0{,}
  \end{equation*}
  Because $\rho^k>0$, the positivity of density is guaranteed if~\eqref{eq:density-BDF2} holds.
\end{proof}


Building on that, by substituting $\hat{\bm u}_1^{k+1}(\vec{\xi}_{i_1i_2i_3}) = J_{i_1i_2i_3} {\bm u}_1^{k+1}$, the internal energy at the $k+1$ pseudotime level can be determined by substituting ${\bm u}_1^{k+1}$ into $\rho E_1^{k+1} = \rho e_1^{k+1} \rho_1^{k+1} + \frac{(m_1^{k+1})^2}{2\rho_1^{k+1}}$ at each solution point, thus leading to the following inequality for $\Delta\tau$ provided that all the conditions of Equation~\eqref{eq:density-BDF2} are satisfied:
\begin{equation}
 \label{eq:internal-energy-BDF2}
 \frac{\rho e_1^{k+1}\rho_1^{k+1}}{{C_2^{\tau}}^2} 
 = A\left(\frac{\Delta\tau}{J}\right)^2 
 + B\frac{\Delta\tau}{J} 
 + C>0,
\end{equation}
where $\rho e_1^{k+1}$ is the total internal energy of ${\bm u}_1^{k+1}$. The coefficients of the quadratic trinomial can be computed using Equation~(\ref{eq:bdf2-update}) as follows:
\begin{equation}
  \begin{array}{ll}
    A &= ({R}_1^{E} {R}_1^{\rho})^k - \frac{1}{2} \left \| ({\bm R}_1^{m})^k \right \|^2
    +\left(\frac{2}{\Delta t}{{\bm u}}_1^{n}-\frac{1}{2\Delta t}{{\bm u}}_1^{n-1}\right)^\top \left[ 
      \begin{array}{l}
      \phantom{-} {R}_1^{E}  \\
      -{\bm R}_1^{m}  \\
      \phantom{-} {R}_1^{\rho} \\
      \end{array}
    \right]^k \\
    &+ \left(\frac{2}{\Delta t}\rho_1^{n}-\frac{1}{2\Delta t}\rho_1^{n-1}\right)
    \left(\frac{2}{\Delta t}(\rho E)_1^{n}-\frac{1}{2\Delta t}(\rho E)_1^{n-1}\right)
    - \left\|\frac{2}{\Delta t}{\bm m}_1^{n}-\frac{1}{2\Delta t}{\bm m}_1^{n-1}\right\|^2, \\ \\
    B &= \left({{\bm u}}_1^{k}\right)^\top\left[ 
      \begin{array}{l}
      \phantom{-} {R}_1^{E}  \\
      -{\bm R}_1^{m}  \\
      \phantom{-} {R}_1^{\rho} \\
      \end{array}
    \right]^k
    + \left(\frac{2}{\Delta t}{{\bm u}}_1^{n}-\frac{1}{2\Delta t}{{\bm u}}_1^{n-1}\right)^\top\left[ 
      \begin{array}{l}
      \phantom{-} \rho E_1   \\
      -{{\bm m}}_1  \\
      \phantom{-} \rho_1 \\
      \end{array}
    \right]^k, \\ \\
    C &= (\rho e)_1^k\rho_1^k,
  \end{array}
\end{equation}
where ${R}_1^{\rho}$, $\bm{R}_1^{m}$, ${R}_1^{E}$ are the right-hand sides of Equation~\eqref{eq:first-order} associated with the continuity, momentum, and energy equations, respectively, and $\left\| \cdot \right\|$ is the Euclidean norm in $\mathbb{R}^3$. Thus, if the solution ${\bm u}_1^k$ at the previous pseudostep is in the admissible set, then the quadratic trinomial has the positive vertical intercept $C$, which implies that there always exists $\Delta \tau:$ $\Delta\tau^{\rho} \ge \Delta \tau > 0$ such that the inequality Equation~(\ref{eq:internal-energy-BDF2}) holds and the present BDF2 DTS scheme preserve the positivity of both internal energy and density. To eliminate the stiffness of the constraint on $\Delta \tau$ required for the positivity of the internal energy in regions where the solution loses its regularity, the entropy-stable velocity and temperature limiters developed in~\cite{upperman2023first} are used. Note that these limiters and the proof of their entropy stability are independent of the temporal discretization and can be directly used for the present BDF2 DTS scheme without any modifications.


\subsection{High-order positivity-violating scheme}
\label{sec:high-order-entropy-stable}

The high-order entropy-stable scheme cannot preserve the positivity of thermodynamic variables because the high-order dissipation operators do not satisfy the maximum principle. Thus, the first-order scheme given by Equation~\eqref{eq:first-order} is combined with its high-order positivity-violating counterpart such that the resultant scheme is positivity preserving, entropy-stable, and high-order accurate in regions where the solution is sufficiently smooth. The positivity-violating entropy-stable scheme is constructed by discretizing the spatial derivatives in RNS Equation~\eqref{eq:brenner-ad-Navier--Stokes} using the high-order spectral collocation operators defined on the same $p$th-order LGL solution points used for the first-order FV scheme. This semi-discrete high-order scheme is given by 
\begin{equation}\label{eq:semi-discrete-ssdc}
  \left(\hat{{\bf u}}_p\right)_t = 
  - 
  [\mathcal{P}^{-1} \Delta]_{d}\hat{\bar{{\bf f}}}^{p(I)}_d 
  + \mathcal{D}_{d} \left[
    \hat{{\bf f}}^{p(V)}_d 
    + \hat{{\bf f}}^{p(AD)}_d 
  \right]
  + 
  \mathcal{P}^{-1}_{d} 
    \hat{{\bf g}}^p_d,
\end{equation} 
where $\hat{{\bf u}}_p = [J] {\bf u}_p$, $\hat{\bar{{\bf f}}}^{p(I)}_d$ for $d=1, 2, 3$ are the $p$th-order contravariant inviscid entropy conservative fluxes defined at the flux points, and $\hat{{\bf g}}^p_d$ includes the boundary, interface and artificial dissipation penalty terms. The full definitions and extended discussion of these fluxes and penalty terms are available in~\cite{carpenter2016entropy,carpenter2014entropy,upperman2022positivity}. 

The fully discrete variant of the above scheme with the implicit BDF2 DTS discretization in the physical time is given by
\begin{equation}
 \label{eq:bdf2-update-high}
 \hat{\bf u}_p^{k+1}
 = C_2^{\tau}\left(\hat{\bf u}_p^{k,n} + \Delta\tau{\bf R}_p^k
 \right).
\end{equation}
This fully discrete high-order spectral collocation scheme is conservative and stable in the entropy sense. {Conservation} follows immediately from the telescopic flux form of the inviscid terms and the SBP form of the viscous and artificial dissipation terms. The entropy-stability of the spatial Navier--Stokes terms in Equation~(\ref{eq:semi-discrete-ssdc}) is proven in~\cite{carpenter2014entropy}, and the entropy dissipation properties of the artificial dissipation terms are shown in~\cite{upperman2022positivity, upperman2023first}.


\subsection{DTS High-order positivity--preserving flux-limiting scheme}
\label{sec:high-order} 

To construct the DTS high-order positivity-preserving entropy-stable scheme, the first-order positivity-preserving \eqref{eq:bdf2-update-first} and high-order positivity-violating \eqref{eq:bdf2-update-high} schemes are combined on each LGL element by using the flux-limiting technique developed in~\cite{yamaleev2023high} as follows:
\begin{align}
  \begin{split}\label{eq:limiting-solution}
    \hat{\bf u}^{k+1}(\theta_f) 
    &= C_2^{\tau}\left(\hat{\bf u}^{k,n} 
    + \Delta\tau \left[
      \theta_f {\bf R}^k_p
      + (1-\theta_f) {\bf R}^k_1
    \right] \right) {,}
  \end{split}
\end{align}
where the flux limiter $\theta^k_f $ $(0 \le \theta^k_f \le 1)$ is a constant on each element \cite{yamaleev2023high} and $\hat{\bf u}^{k+1}_p = \hat{\bf u}^{k+1}|_{\theta_f=1}$ and $\hat{\bf u}^{k+1}_1 = \hat{\bf u}^{k+1}|_{\theta_f=0}$ are $p$th- and 1st-order numerical solutions, respectively, which are defined on the same LGL elements with the same high-order metric terms.

Since $0<C_2^{\tau}<1$, a proof that the high-order scheme given by Equation~(\ref{eq:limiting-solution}) guarantees pointwise positivity of density and temperature is nearly identical to that presented in~\cite{yamaleev2023high} for the same flux-limiting entropy-stable scheme with the explicit Euler discretization in the physical time. The entropy-stability of the flux-limiting scheme \eqref{eq:limiting-solution} follows immediately from the fact that this hybrid scheme is a linear convex combination of two entropy-stable schemes on each high-order element. Further details of the positivity and entropy-stability of the high-order flux-limiting scheme can be found in~\cite{yamaleev2023high}.


\section{Numerical Results}
\label{sec:results}

The accuracy and efficiency of the present DTS BDF2 spectral collocation scheme and the SSPRK3 scheme developed in~\cite{yamaleev2023high} are compared on two benchmark supersonic unsteady viscous flow problems. For all test problems considered, the flow quantities are non-dimensionalized as follows: $t=\frac{t'}{L^*/v^*}$, $x_j=\frac{x_j'}{L^*}$, $\rho=\frac{\rho'}{\rho^*}$, $v_i=\frac{v_i'}{v^*}$, $\mu=\frac{\mu'}{\mu^*}$, $p=\frac{p'}{p^*}$, $R=\frac{R'}{R^*}$, $E=\frac{E'}{\gamma c_{\rm v}^* T^*}$, $\kappa=\frac{\kappa'}{\kappa^*}$ and $c_{\rm v}=\frac{c_{\rm v}'}{c_{\rm v}^*}$. Thus, the non-dimensional pressure, total and kinetic energy variables are given by
\begin{align}
    &p=\rho T, &E=\frac{1}{\gamma} T + E_k, &&E_k=\frac{(\gamma-1)M_{\infty}^2}{2}u^2.
\end{align}

Convergence criteria of the BDF2 DTS inner-loop are based on the relative and absolute errors that are defined as follows
\begin{align}
    &\epsilon_{abs}
    = \frac{\|{\bf u}^{k+1}-{\bf u}^k\|_{L^2}}{\Delta\tau},
    &&\epsilon_{rel}
    = \frac{\|{\bf u}^{k+1}-{\bf u}^k\|_{L^2}}
    {\|{\bf u}^{1}-{\bf u}^0\|_{L^2}}
    \frac{\Delta\tau_0}{\Delta\tau}.
\end{align}


\subsection{2D cylinder flow at $\mathbf{M_{\infty}=17.605}$}
\label{subsec:hypersonic-cylinder}


\begin{figure}[!t]
    \centering
    \includegraphics[width=0.9\textwidth]{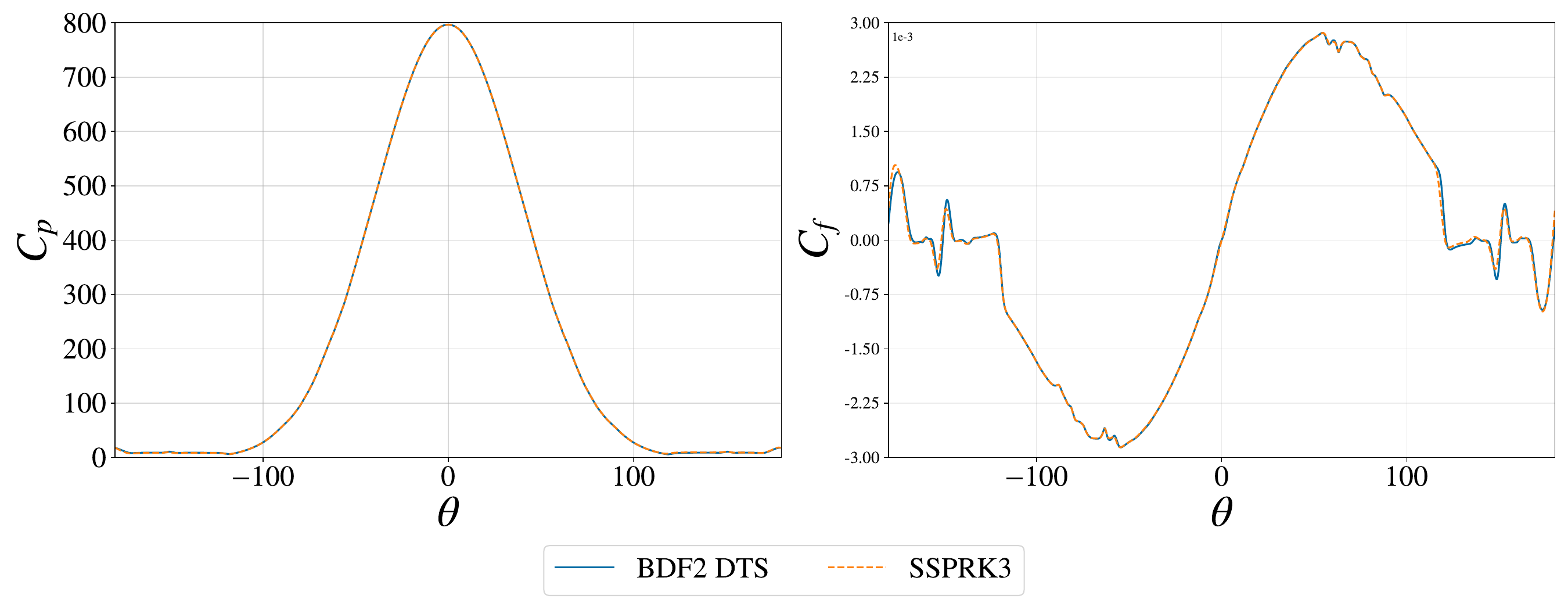}
    \caption{Time-averaged wall pressure (left panel) and skin friction coefficient computed using the BDF2 DTS and the explicit SSPRK3 $p=5$ schemes for the hypersonic cylinder flow at $M_{\infty}=17.605$ on the $N_\text{elem}=55,216$ element grid.}
    \label{fig:hypcyl-55k-p5-comparison}
\end{figure}


\begin{figure}[!t]
    \centering
    \begin{minipage}{0.45\textwidth}
        \centering
        \includegraphics[width=\textwidth]{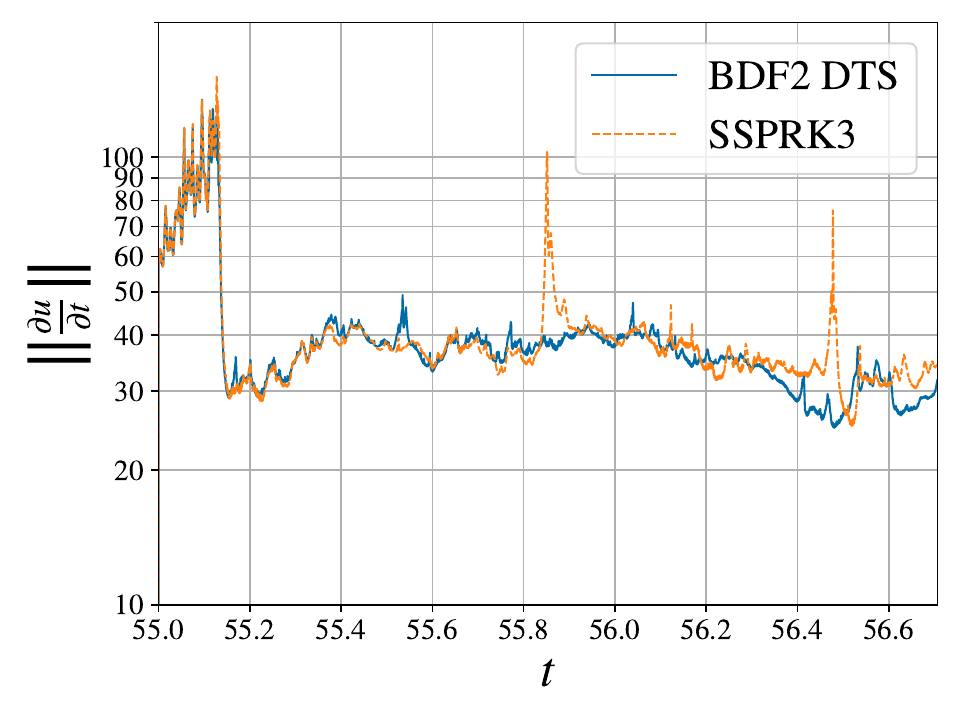}
    \end{minipage}
    \hfill
    \begin{minipage}{0.45\textwidth}
        \caption{Time histories of the residual norm obtained with the BDF2 DTS and explicit SSPRK3 $p=5$ schemes for the hypersonic cylinder flow at $M_{\infty}=17.605$ on the $N_\text{elem}=55,216$ element grid.}
        \label{fig:hypcyl-55k-p5-dudt}
    \end{minipage}
\end{figure}


\noindent
The first test problem is the hypersonic flow around a cylinder, which is a standard benchmark problem for verifying robustness and accuracy of high-order numerical schemes for simulating compressible viscous flows with strong discontinuities, where the positivity of thermodynamic variables plays a critical role. The 2D hypersonic flow around a cylinder is simulated at $M_{\infty}=17.605$ and $Re_{\infty}=376,930$. The cylinder center is located at $(x,y)=(0,0)$, and its radius is set equal to $r=0.5$. The computational domain is a rectangle: $-2\leq x\leq 3$ and $-2\leq y\leq 2$. The supersonic inflow and outflow boundary conditions are imposed on the left and right boundaries, respectively, while the supersonic freestream boundary conditions are used at the top and bottom boundaries. The entropy-stable no-slip boundary conditions developed in~\cite{dalcin2019conservative} are imposed on the cylinder wall. For this test problem, a grid with $N_{\rm elem}=55,216$, $p=5$ elements is used. This grid is stretched in the direction normal to the cylinder wall so that the wall grid spacings is $\Delta r = 1.33\times 10^{-3}$. The flow is initialized with the constant freestream flow, whose velocity vector is gradually reduced to zero at the cylinder wall. The problem is solved numerically using the implicit BDF2 DTS and explicit SSPRK3 schemes. The test case is run until $t=55$, corresponding to the time when the wake becomes fully developed. After that, the test problem is integrated for additional $5$ nondimensional time units to compute time-averaged quantities.

The time-averaged wall pressure and skin friction coefficients computed using the implicit BDF2 DTS and explicit SSPRK3 schemes are compared in Figure~\ref{fig:hypcyl-55k-p5-comparison}. As shown in the figure, the pressure coefficient obtained with the present scheme is practically identical to that of the SSPRK3 solution. Practically, the same accuracy is achieved for the skin friction coefficient, which is much more sensetive quantity because it depends on the solution gradient. Though the physical time step size of the BDF2 scheme is on average $20$ times higher than that of the explicit counterpart. 

Time histories of the $L_2$ residual norms of the BDF2 DTS and SSPRK3 schemes are compared in Figure~\ref{fig:hypcyl-55k-p5-dudt}. As can be see in this figure, the BDF2 residual norm closely follows that of the SSPRK3 scheme with the exception of two pronounced peaks at $t=55.85$ and $56.45$. These deviations can be explained by the much larger time step size used in the implicit scheme compared with its explicit counterpart.

The present BDF2 DTS scheme provides nearly the same accuracy as the explicit third-order scheme for this hypersonic viscous flow with strong discontinuities. Note that for this test problem, the tolerence used for determining convergence to the steady-state in the pseudotime varies in the interval from $\epsilon_\text{rel}=2 \times 10^{-2}$ to $5 \times 10^{-2}$. As seen in Figures~\ref{fig:hypcyl-55k-p5-comparison}-\ref{fig:hypcyl-55k-p5-dudt}, this convergence error in the pseudotime has no appreciable effect on the solution accuracy in the physical time.


\subsection{3D supersonic Taylor-Green vortex flow}
\label{subsec:taylor-green-vortex}


\begin{figure}[!t]
    \centering
    \includegraphics[width=0.9\textwidth]{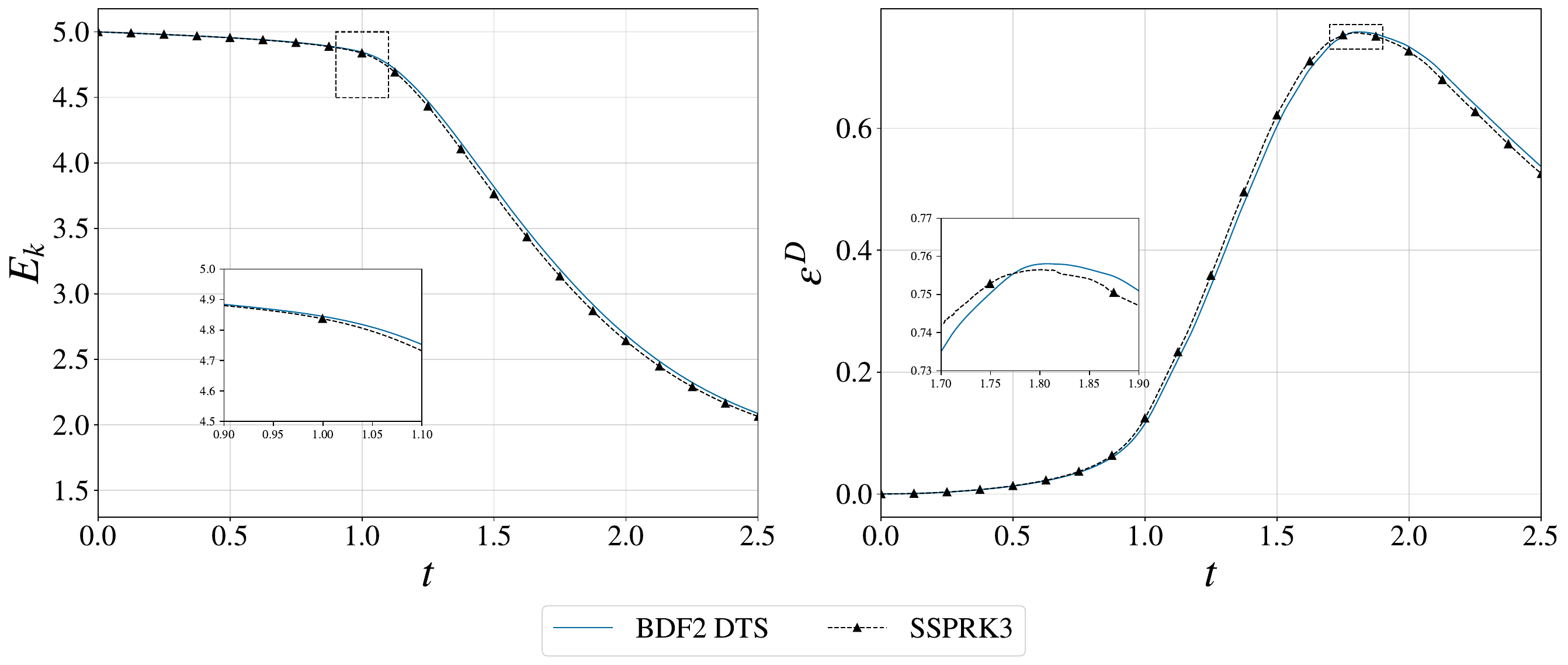}
    \caption{Time histories of the total kinetic energy (left-panel) and dilation computed with the BDF2 DTS and SSPRK3 $p=6$ schemes on the $N_\text{elem}=64^3$ grid for the TGV flow at $Re_{\infty}=400$ and $M_{\infty}=10.0$.}
    \label{fig:tgv-Re400-Ma10}
\end{figure}


\begin{figure}[!t]
    \centering
    \includegraphics[width=0.9\textwidth]{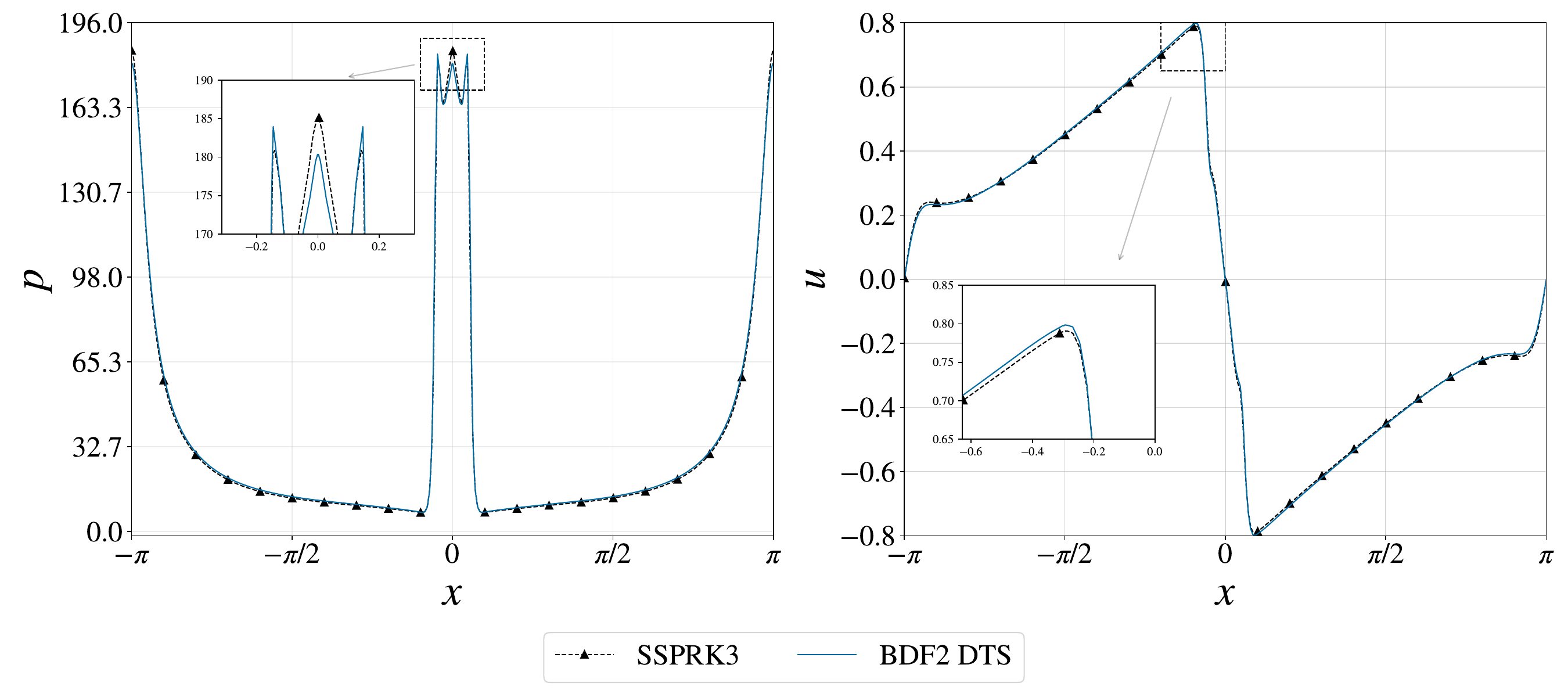}
    \caption{Pressure (left-panel) and $x$-component of the velocity vector profiles along the line $y=\pi$ and $z=0$ obtained with the BDF2 DTS and SSPRK3 $p=6$ schemes on the $N_\text{elem}=64^3$ grid for the TGV flow at $Re_{\infty}=400$ and $M_{\infty}=10.0$.}
    \label{fig:tgv-Re400-Ma10-primitives}
\end{figure}


\noindent
The second test problem is the 3D Taylor--Green vortex (TGV) flow at the Reynolds number of $Re_{\infty}=400$ and Mach number $M_{\infty} = 10$. This problem is considered to assess the performance of the present BDF2 DTS scheme for the essentially unsteady turbulent flow with strong discontinuities. This problem is solved on the periodic cube ($-\pi \leq x,y,z \leq \pi$) with the following initial conditions: 
\begin{align}
    \begin{split}\label{eq:tgv-Re400-init}
        p(x,y,z) &= 1 + \frac{\gamma Ma^2}{89.6}(\cos 2x + \cos 2y) (\cos 2z + 2),\\
        V_x(x,y,z) &= \sin x \cos y \cos z,\\
        V_y(x,y,z) &= - \cos x \sin y \cos z,\\
        V_z(x,y,z) &= 0,\\
        T(x,y,z) &= 1.
    \end{split}
\end{align}
The initial density state is computed as $\rho(x,y,z) = p/T$. The following quantities are measured to evaluate accuracy and convergence properties of the BDF2 DTS scheme:  
\begin{equation}
    E_k = \frac{(\gamma-1)Ma^2}{|\Omega|}\int_\Omega\frac{\rho v_i v_i}{2} \text{d}\Omega,
\end{equation}
\begin{equation}
\label{dilation}
    \varepsilon^D = \frac{4(\gamma-1)Ma^2}{3Re|\Omega|}\int_\Omega\mu(T)\left(\frac{\partial v_j}{\partial x_j}\right)^2 \text{d}\Omega,
\end{equation}
where  $E_k$ is a total kinetic energy, and $\varepsilon^D$ is the dilational contributions to the viscous dissipation rate of the kinetic energy.

Since no reference solutions or experimental data are available in {the} literature for the $Ma_{\infty}=10$ case, the results of the BDF2 DTS scheme are compared with those of the positivity-preserving SSPRK3 scheme developed in~\cite{yamaleev2023high}. Figure~\ref{fig:tgv-Re400-Ma10} shows time histories of the total kinetic energy and the dilational component of the kinetic energy dissipation rate computed with the BDF2 and SSPRK3 schemes. The results obtained with the BDF2 DTS and SSPRK3 $p=6$ schemes on the $N_\text{elem}=64^3$ grid are nearly identical. This comparison shows that the present BDF2 DTS scheme provides high temporal accuracy for gradient quantities such as the dilational component, which is very sensitive to strong shock waves and their interaction with vortices. In Figure~\ref{fig:tgv-Re400-Ma10-primitives}, snapshots of the pressure and $x$-component of the velocity vector profiles along the line $(x,y,z)=(x,\pi,0)$ are compared. This figure shows that the BDF2 DTS solution is slightly more dissipative than that computed with the SSPRK3 on the same grid.


\begin{table}[!t]
    \centering
    \caption{The total number of explicit steps and wall clock time average over three runs required for the BDF2 DTS and SSPRK3 positivity-preserving schemes to advance the numerical solution over $0.1$ time units in the physical time on the the $N_\text{elem}=64^3$ $p=6$ grid for the TGV flow at $Re_{\infty}=400$ and $M_{\infty}=10.0$.}
    \label{tab:tgv-run-stats}
    \begin{tabular}{p{0.3\textwidth}|p{0.2\textwidth}p{0.2\textwidth}c}
        \centering Solver         
        &\centering SSPRK3
        &\centering BDF2 DTS              
        & \\
        \hline
        \centering $\Delta t$     
        &\centering $\sim5\times 10^{-5}$
        &\centering $2\times 10^{-3}$             
        & \\
        \centering $\Delta\tau$   
        &\centering --                    
        &\centering $\sim6\times 10^{-5}$ 
        & \\
        \centering Explicit $(\hat{\bm u})_t$ calls 
        &\centering $5925$                
        &\centering $5249$                
        & \\
        \centering Wall clock time       
        &\centering $66,683s$           
        &\centering $48,619s$            
        & \\
    \end{tabular}
\end{table}


To assess the efficiency of the proposed DTS BDF2 scheme as compared with its explicit counterpart, both schemes are run on the same $N_{\rm elem} = 64^3$ and $p = 6$ grid for an additiona $0.1$ nondimensional time units starting from the final time $t=2.5$. These schemes were run on the Old Dominion University's Turing cluster at a fixed 256-core on exclusive coreV3 nodes to eliminate resource contention and isolate algorithmic performance. The total number of explicit steps and the wall clock time are measured to compare the efficiency of both schemes, which are summarized in Table~\ref{tab:tgv-run-stats}. 

The BDF2 scheme uses a fixed time step of $\Delta t = 2\times 10^{-3}$. For the SSPRK3 scheme, the time step is controlled by the positivity-preserving conditions, and an average of $\sim 5\times 10^{-5}$ was observed over the time period. The total number of explicit steps for the BDF2 DTS scheme is $5249$, while for the SSPRK3 scheme it is $5925$. The wall clock time average over three runs for the BDF2 DTS and SSPRK3 schemes are $48,619$ and $66,683$ seconds, respectively. Thus, the BDF2 DTS scheme demonstrates an $11.4\%$ reduction in the number of explicit steps and a $27\%$ reduction in the wall clock time relative to the SSPRK3 scheme, while maintaining the comparable accuracy. {Since the positivity-preserving mechanisms are only active when necessary, the extra $11\%$ RHS calls are not uniform in wall-clock time throughout the simulation and may raise the cost disproportionately.}

\begin{remark}
    A reasonable question is to ask whether the second-order SSPRK2 scheme would be more efficient than the BDF2 DTS scheme. The SSPRK2 scheme was also run on the same grid using the exact settings. The total number of explicit steps for the SSPRK2 scheme is $3868$, while the wall clock time average over three runs is $51,648$ seconds. When compared to the BDF2 DTS scheme, the SSPRK2 scheme shows a $\%5.8$ increase in the wall clock time. Therefore, it is still less efficient than the BDF2 DTS scheme without assessing its accuracy.
\end{remark}


\section{Conclusions}
\label{sec:conclusions}


In this paper, the explicit SSPRK3 positivity-preserving entropy-stable spectral collocation schemes of arbitrary spatial order of accuracy introduced in \cite{yamaleev2023high} for the 3D compressible Navier-Stokes equations {are extended} to an implicit DTS formulation based on the BDF2 time integrator. The proposed dual time-stepping scheme combines unconditional stability properties of the implicit BDF2 time integrator with the positivity-preserving and entropy stability properties of the baseline explicit spectral collocation scheme, {while providing design-order accuracy in physical time.} This DTS methodology guarantees the positivity of thermodynamic variables at each pseudotime iteration and imposes no constraints on the physical time step size. The accuracy and efficiency of the present DTS BDF2 positivity-preserving entropy-stable spectral collocation scheme are assessed using two benchmark problems: the hypersonic flow around a cylinder and the 3D supersonic TGV flow. The results obtained with the present BDF2 DTS scheme are compared with those of the explicit SSPRK3 scheme for the same spatial order of accuracy and grids. For both test problems, the BDF2 DTS scheme provides accuracy comparable to that of the SSPRK3 scheme. Furthermore, for the 3D supersonic TGV flow, the BDF2 DTS scheme provides close to $27\%$ reduction in the wall clock time as compared with the SSPRK3 scheme without sacrificing the solution accuracy.



\section*{Acknowledgments}
The second author gratefully acknowledges the support from Department of Defense through grant W911NF2310183. 


\bibliographystyle{spmpsci}
\bibliography{references}

\end{document}